\documentclass[12pt]{extarticle}

\usepackage{calrsfs}
\usepackage{amsmath}
\usepackage{amsfonts}
\usepackage{latexsym}
\usepackage{amssymb}
\usepackage{enumerate}
\usepackage{amsthm}
\usepackage[normalem]{ulem}
\usepackage{csquotes}
\usepackage{lineno,xcolor}
\usepackage[normalem]{ulem}
\usepackage{filecontents}
\usepackage{scalefnt}
\usepackage{tikz}
\usepackage{fancyhdr}
\usetikzlibrary{arrows}
\usetikzlibrary{positioning}

\usepackage{lineno,xcolor}

\usepackage{hyperref}
\hypersetup{
    colorlinks,
    linkcolor={red!60!black},
    citecolor={green!60!black},
    urlcolor={blue!60!black}
}

\definecolor{amaranth}{rgb}{0.9, 0.17, 0.31}
\definecolor{bluegray}{rgb}{0.4, 0.6, 0.8}

\makeatletter

\newtheorem*{maintheorem*}{Main Theorem}
\newtheorem{theorem}{Theorem}[section]
\newtheorem{proposition}[theorem]{Proposition}
\newtheorem{corollary}[theorem]{Corollary}
\newtheorem{lemma}[theorem]{Lemma}

\newtheorem*{theorem*}{Theorem}

\newtheorem{remark}[theorem]{Remark}

\newtheorem*{example*}{Example}
\newtheorem*{conjecture*}{Conjecture}
\def\1{\mathbf 1}

\def\u{\mathbf u}
\def\vv{\mathbf v}

\def\0{\mathbf 0}

\def\cB{\mathcal B}
\def\cC{\mathcal C}
\def\cD{\mathcal D}

\def\cF{\mathcal F}
\def\cG{\mathcal G}
\def\cH{\mathcal H}

\def\cK{\mathcal K}
\def\cO{\mathcal O}
\def\cP{\mathcal P}

\def\cS{\mathcal S}

\def\PG{{\rm PG}}

\def\PGammaL{{\rm P\Gamma L}}
\def\AGammaL{{\rm A\Gamma L}}

\def\GF{{\rm GF}}
\def\GL{{\rm GL}}

\def\Aut{{\rm Aut}}

\def\GF{{\rm GF}}

\def\mod{{\rm mod} }

\def\Tr{{\rm Tr}}

\def\<{\langle}
\def\>{\rangle}
\newcommand\comment[1]{}

\newcommand*{\shifttext}[2]{
  \settowidth{\@tempdima}{#2}
  \makebox[\@tempdima]{\hspace*{#1}#2}
}

\newcommand\redsout{\bgroup\markoverwith{\textcolor{amaranth}{\rule[0.5ex]{2pt}{0.4pt}}}\ULon}

\newcommand\redout{\bgroup\markoverwith
{\textcolor{red}{\rule[.4ex]{2pt}{0.8pt}}}\ULon}

\title{Classification of spreads \\ of Tits quadrangles of order $64$}
\author{Giusy Monzillo\footnote{The research was supported by the Italian National Group for Algebraic and Geometric Structures and their Applications (GNSAGA-INdAM). } \\
\small {\tt giusy.monzillo@famnit.upr.si}\\[0.8ex]
\small UP FAMNIT\\[-0.8ex]
\small University of Primorska\\[-0.8ex]
\small Glagolj\v aska 8 \\[-0.8ex]
\small 6000 Koper, Slovenia
\and   
Tim Penttila \\ 
\small {\tt tim.penttila@adelaide.edu.au}\\
\small School of Mathematical Sciences\\[-0.8ex]
\small The University of Adelaide\\[-0.8ex]
\small Adelaide, South Australia \\[-0.8ex]
\small 5005 Australia
\and
Alessandro Siciliano$^*$ \\
\small{\tt alessandro.siciliano@unibas.it}\\[0.8ex]
\small Dipartimento di Matematica, Informatica ed Economia\\[-0.8ex]
\small Universit\`a degli Studi della Basilicata\\[-0.8ex]
\small Viale dell'Ateneo Lucano 10 \\[-0.8ex]
\small 85100 Potenza, Italy\\
}
\date{}

\begin{document}


\maketitle


\begin{abstract}
Brown et al. provide a representation of a spread of the Tits quadrangle $T_2(\cO)$,  $\cO$ an oval of $\PG(2,q)$,  $q$ even, in terms of a certain family of $q$ ovals of $\PG(2,q)$. 
By combining this representation with the Vandendriessche classification of hyperovals in $\PG(2,64)$ and the  classification of flocks of the quadratic cone in $\PG(3,64)$, recently given by the authors, in this paper, we classify all  the spreads of $T_2(\cO)$, $\cO$ an oval of $\PG(2,64)$, up to equivalence. These complete the classification  of spreads of $T_2(\cO)$ for  $q\le 64$.
\end{abstract}



\section{Introduction}

	The recent classification of hyperovals of $\PG(2,64)$ \cite{vander} has provided an opportunity to extend classification results for related structures to order 64. The first of these was inversive planes \cite{pen}, which correspond to fans of ovals. The second of these was flocks (and the related elation generalized quadrangles) \cite{mps}, which correspond to herds of ovals. This paper gives a third: spreads of Tits quadrangles, which correspond to generalized fans of ovals.
	
	This paper is thus a sequel to \cite{boppr}, which classified spreads of Tits quadrangles of order at most 32. We take the opportunity to clarify some issues of when isomorphic spreads of Tits quadrangles correspond to isomorphic generalized fans of ovals, left open in \cite{boppr1}, by introducing the concept of a posy of ovals (and of equivalence of posies) in Section \ref{sec_2}. We also correct an error in a line of a table in \cite{boppr}.

All spreads of Tits quadrangles of order 64 were previously known -- no new spreads were discovered in the course of this classification. 
Of the nineteen Tits quadrangles of order 64, fifteen admit a unique spread. Each of the remaining four, which are already known to be proper subquadrangles of a generalized quadrangle with the same number of lines on a point, has a spread on every point of the oval  on which it is constructed; all these spreads are subtended. In addition, these Tits quadrangles are actually the {\em only} subquadrangles of a generalized quadrangle with the same number of lines on a point; the Tits quadrangle arising from the pointed conic is not among the latter, as  is also the case for orders 16 and 32.

\section{Preliminaries and known results}\label{sec_2}

A (finite) {\em generalized quadrangle} (GQ) is an incidence structure $\cS=(\cP,\cB, {\rm I})$ where  $\cP$ and $\cB$ are disjoint (nonempty) sets of objects called {\em points} and {\em lines}, respectively, and for which $\rm I\subseteq (\cP\times\cB)\cup (\cB\times\cP)$ is a symmetric point-line incidence relation satisfying the following axioms:
\begin{itemize} 
\item[ (i)] Each point is incident with $t+1$ lines, and two distinct points are incident with at most one line.
\item[(ii)]  Each line is incident with $s+1$ points, and two distinct lines are incident with at most one point.
\item[(iii)] If $x$ is a point and $L$ is a line not incident with $x$, then there is a unique pair $(y, M)\in \cP\times \cB$ such that $x\, {\rm I}\, M\, {\rm I}\, y\, {\rm I}\,  L$.
\end{itemize}

The integers $s$ and $t$ are the {\em parameters} of the GQ, and $\cS$ is said to have {\em order} $(s,t)$; if $s=t$, it is said to have order $s$.  If $\cS$ has order $(s,t)$, then it follows that $|\cP| = (s + 1)(st + 1)$ and $|\cB| = (t + 1)(st + 1)$. If $\cS=(\cP,\cB, {\rm I})$ is a GQ of order $(s,t)$, then the incidence structure $\cS^*=(\cB,\cP, {\rm I})$ is a GQ of order $(t,s)$, and it is called the {\em dual of}  $\cS$. A {\em spread}  of $\cS$ is a set $\Sigma$ of lines  of $\cS$ such that every point of $\cS$ is incident with exactly one element of $\Sigma$.  An {\em ovoid}  of $\cS$ is a set $\Omega$ of points of $\cS$ such that every line of $\cS$ is incident with exactly one element of $\Omega$. It is evident that, under duality, spreads/ovoids of $\cS$ correspond to ovoids/spreads of $\cS^*$. If $\Sigma$ is a spread and $\Omega$ is an ovoid of the GQ $\cS$ of order $(s,t)$, then $|\Sigma|=|\Omega|=st+1$. Let  $\Sigma_i$ be a spread of a GQ $\cS_i$, $i=1,2$. Then, $\Sigma_1$ and $\Sigma_2$ are said to be {\em equivalent} if there is an isomorphism from $\cS_1$ to $\cS_2$ which maps $\Sigma_1$ to $\Sigma_2$. The same definition applies to ovoids.  Spreads and ovoids of generalized quadrangles are surveyed in \cite{tp}.

The classical generalized quadrangles of order $q$ are $W(q)$, whose points and lines are those of a symplectic geometry in $\PG(3,q)$; and $Q(4,q)$, whose points and lines are those of a non-singular quadric in $\PG(4,q)$. It is known that $Q(4,q)$ is isomorphic to the dual of $W(q)$, and it is self-dual when $q$ is even \cite{pt}. For more details on generalized quadrangles, the reader is referred to \cite{pt}.

An {\em oval} in $\PG(2,q)$ is a set $\cO$ of $q+1$ points, no three collinear. A line $\ell$ of $\PG(2,q)$ is said to be {\em external}, {\em tangent} or {\em secant} to $\cO$ according as it intersects $\cO$ in 0, 1, or 2 points. When $q$ is even, all the tangent lines to $\cO$ meet in a common point $N$ called the {\em nucleus} of $\cO$. The set $\cO\cup\{N\}$ is a {\em hyperoval}, that is, a set of $q+2$ points, no three collinear. It is well known that hyperovals exist only when $q$ is even. Given  an oval $\cO$ in $\PG(2,q)$, $q$ even, it is possible to choose projective coordinates such that $\cO$ contains the points $\{(1,0,0), (1, 1, 1), (0, 0, 1)\}$\footnote{\ To be formally correct, $(x,y,z)$ should be replaced by $\<(x,y,z)\>$ to denote the point defined by the vector $(x,y,z)$, but we guess that the only result of putting angle brackets everywhere in the paper would be to burden it with a uselessly heavy notation. For this reason, we warn the reader that throughout this paper we will freely use this simplified notation.}. This yields that it can be written as
\[
\cO = \cD(f)= \{(1,t,f(t)):t\in \GF(q)\} \cup \{(0,0,1)\},\ \hbox{with nucleus}\ (0,1,0),
\]
where $f$ is a permutation of $\GF(q)$ with $f(0) = 0$, $f(1) = 1$, and such that $f_s:x\mapsto (f(x+s)+f(s))/x$ is a permutation of $\GF(q)\setminus\{0\}$, for any $s\in\GF(q)$.
Permutations of $\GF(q)$ with these properties are called {\em o-polynomials} \cite[p.185]{hir}. It turns out that the degree of an o-polynomial is at most $q-2$. If  $f(1)=1$ is not required but the other conditions are, then $f$ is an {\em o-permutation over} $\GF(q)$. 

Conversely, for every polynomial $f$ such that $f(0) = 0$, $f(1) = 1$ and $f_s:x\mapsto (f(x+s)+f(s))/x$ is a permutation of $\GF(q)\setminus\{0\}$, for any $s\in\GF(q)$, the point-set $\cD(f)=\{(1,t,f(t)):t\in\GF(q)\}\cup \{(0,0,1)\}$ is an oval in $\PG(2,q)$ with nucleus $(0,1,0)$ \cite[Theorem 8.22]{hir}.

In this setting, the oval $\cD(x^2)$ is a {\em conic}, while  $\cD(x^{1/2})$ is a {\em pointed conic} of $\PG(2,q)$.
 
%
  

%
A {\em translation oval} $\cO$ in $\PG(2,q)$, $q$ even, is an oval which is invariant under a group $E$ of elations of order $q$ such that all the elations in $E$ have a common axis $\ell$.  The line $\ell$ is a tangent to $\cO$ and is called an {\em axis} of $\cO$. Payne \cite{payne} proved that every translation oval can be written in the form
\[
\cD(\alpha)=\cD(t^\alpha)= \{(1,t,t^{\alpha}):t\in\GF(q)\} \cup \{(0,0,1)\},
\]
for some  generator $\alpha$ of $\Aut(\GF(q))$.

Starting from an oval $\cO$ in $\PG(2,q)$ embedded as a plane in $\PG(3,q)$, Tits constructed a class of (non-classical) GQs of order $q$ \cite{dem,pt}. \\ The points are:  
\begin{itemize}
\item[(i)] the  points of $\PG(3,q)$ not in $\PG(2,q)$;
\item[(ii)] the  planes of $\PG(3,q)$ meeting $\PG(2,q)$ in a unique point of $\cO$;
\item[(iii)] the symbol $(\infty)$.
\end{itemize}

The lines are: 
\begin{itemize}
\item[(a)] the lines of $\PG(3,q)$ not in $\PG(2,q)$ meeting $\PG(2,q)$ in a point of $\cO$;
\item[(b)] the points of $\cO$. 
\end{itemize}

The incidence relation is as follows: a point of type (i) is incident only with the lines of type (a) which contain it;  a point of type (ii) is incident with all lines of type (a) contained in it and  with the unique line of type (b) on it;   the point  of type (iii) is incident with   all lines of type (b) and no line of type (a).
This GQ is denoted by $T_2(\cO)$.

In the following, whenever an oval $\cO$ in $\PG(2,q)$ is considered, the corresponding Tits quadrangle $T_2(\cO)$ is constructed by embedding $\PG(2,q)$ as the plane $\pi_\infty$ in $\PG(3,q)$ with equation $x_0=0$.

\begin{theorem}\cite[Theorem 3.1]{boppr1}\label{th_4}
Let $\cO$ be an oval in $\PG(2,q)$ and let $T_2(\cO)$ be the Tits quadrangle constructed from $\cO$. Then,
\begin{itemize}
\item[(i)] if $\cO$ is not a conic, then $\Aut(T_2(\cO))$ is the stabilizer of $\cO$ in $\PGammaL(4,q)$;
\item[(ii)] if $\cO$ is a conic, then the stabilizer of $(\infty)$ in $\Aut(T_2(\cO))$  is the stabilizer of $\cO$ in $\PGammaL(4,q)$.
\end{itemize}
\end{theorem}
\begin{corollary}\cite[Corollary 3.2]{boppr1}\label{cor_2}
	Let $\cO$, $\cO'$ be ovals in $\PG(2,q)$ and $T_2(\cO)$, $T_2(\cO')$  the Tits quadrangles constructed  from $\cO$, $\cO'$, respectively. Then,
\begin{itemize}
\item[(i)] if $\cO$ is not a conic, every isomorphism between $T_2(\cO)$ and $T_2(\cO')$ is induced by a collineation of $\PG(3,q)$ which maps $\cO$ to $\cO'$, and conversely;
\item[(ii)] if $\cO$ is a conic, with $(\infty)$ and $(\infty)'$ the points of type (iii) of $T_2(\cO)$ and $T_2(\cO')$, respectively, then every isomorphism between $T_2(\cO)$ and $T_2(\cO')$ mapping $(\infty)$ to $(\infty)'$ is induced by a collineation of  $\PG(3,q)$ which maps $\cO$ to $\cO'$, and conversely.
\end{itemize}
\end{corollary}

By a result in Payne and Thas \cite[3.2.2]{pt},  $T_2(\cO)$ is isomorphic to $Q(4,q)$ if and only if $\cO$ is a conic. When $q$ is even and $\cO$ is a conic, $T_2(\cO)$ is 
isomorphic to $W(q)$, and conversely; in this case $T_2(\cO)$ is said to be {\em classical} . When $q$ is odd, $T_2(\cO)$ has no spreads, since  every oval of $\PG(2,q)$ is a conic, by the well-known theorem by Segre, and $Q(4,q)$, $q$ odd, has no spreads \cite[3.4.1(i)]{pt}. Spreads of $T_2(\cO)$ exist for $q$  even, and they consist of one line of type (b) and $q^2$ lines of type (a). 

From now on let $q$ be even. 
If $\cO$ is a conic $\cC$, spreads of  $T_2(\cC)$ correspond bijectively to ovoids of $W(q)$, i.e., ovoids in $\PG(3,q)$ \cite[p.54]{dem},\cite{thas72}. There are at least two non-equivalent ovoids in $\PG(3,q)$: the elliptic quadric, which exists for $q=2^h$, $h\ge1$, and the Tits ovoid, that exists for $q=2^h$, $h\ge3$ odd. Therefore, there are at least two non-equivalent spreads of $T_2(\cC)$, when $q$ is an odd power of 2. According to  \cite[12.5.2]{pt}, $T_2(\cO)$ is self-dual  if and only if  $\cO$ is a translation oval; so, in this case,  spreads of  $T_2(\cO)$ correspond bijectively to ovoids of $T_2(\cO)$. By using this equivalence, some spreads of $T_2(\cO)$ have been constructed. 

 Let $\pi$ be a plane of $\PG(3,q)$ intersecting $\pi_\infty$ in an external line $\ell$ to $\cO$. The point set $(\pi\setminus \ell)\cup\{(\infty)\}$ is an ovoid of $T_2(\cO)$ known as {\em planar ovoid};  when $\cO$ is a translation oval, it produces  a  spread of $T_2(\cO)$.
 
 If $\cO$ is a conic, then this spread is the one that corresponds to the elliptic quadric of $\PG(3,q)$. 
 If $\cO$ is a translation oval and $q=2^h$, $h$ odd, then $T_2(\cO)$ is also self-polar \cite[12.5.2]{pt}, and the absolute lines of the polarity form a spread \cite[1.8.2]{pt}.

 Let $q=2^{2e+1}$, $e\ge1$, and $\sigma:x\mapsto x^{2^{e+1}}\in\Aut(\GF(q))$. Let $\Omega$ be a Tits ovoid  in $\PG(3,q)$ such that $\pi_\infty$ meets $\Omega$ in an oval $\cO$. Then, $\cO$ is equivalent to $\cD(\sigma)$ \cite{tits}. It follows that the points of $\Omega\setminus\cO$ together with the tangent planes to $\Omega$ at a point of $\cO$ form an ovoid of $T_2(\cO)$, and hence it gives rise to a spread of $T_2(\cO)$ by duality\footnote{In the Table on page 280 of \cite{boppr1}, the o-polynomial describing $\cD(\sigma)$ should be $f(x)=x^\sigma$ and $\cO_s=\{(1,t,t^\sigma+s^{(\sigma+2)/\sigma-1}):t \in \GF(q)\}\cup \{(0,0,1)\}$.}.

 Spreads of $T_2(\cO)$ can be also obtained from flock GQs, i.e., GQs arising from flocks of the quadratic cone in $\PG(3,q)$. We direct the reader to \cite{pay96} for definitions and references on flock GQs. Let $\cS$ be a flock GQ of order $(q^2,q)$ ($q$ even). Then, $\cS$ contains a family of subquadrangles of order $q$,  each of which is of type $T_2(\cO)$, for some oval $\cO$. These results were originally obtained in \cite{pay85,pm}, but also appear in the more accessible \cite{boppr}.

  Fix such a subquadrangle $T_2(\cO)$ and let $\ell$ be a line of $\cS$ but not of $T_2(\cO)$. Then, $\ell$ is disjoint from $T_2(\cO)$ and every point on $\ell$ is incident with exactly one (extended) line of $T_2(\cO)$. This set of $q^2+1$ lines forms a spread of $T_2(\cO)$. 
By \cite[2.2.1]{pt}, a similar construction works every time  $T_2(\cO)$ is a proper subquadrangle of a GQ of order $(s,q)$, $s\le q^2$; such a spread is  said to be {\em subtended}.

It is possible to construct spreads of $T_2(\cO)$ from a given one, as shown in \cite{boppr}.
Let $N$ be the nucleus of $\cO$ and $\Sigma$ a spread of $T_2(\cO)$ containing a point $P$ of $\cO$. Then, the set $\Sigma'=(\Sigma\cup \{N\})\setminus\{P\}$ is a spread of $T_2((\cO\cup \{N\})\setminus\{P\})$; this method of constructing spreads is called {\em nucleus swapping} \cite{boppr1} (This operation is called {\em nucleus switching} in \cite{boppr}).

In \cite{boppr1}, the equivalence between spreads of $T_2(\cO)$ and a certain family of ovals in $\PG(2,q)$ is established.
\\ Let $\cO_1$ and $\cO_2$ be ovals in $\PG(2,q)$,  and $P$ a point not in $\cO_1\cup \cO_2$. The ovals $\cO_1$ and $\cO_2$ are said to be {\em compatible at} $P$ if they have the same nucleus $N$, a point $Q$ in common, the line $\<P,Q\>$ as a common tangent and if every secant line to $\cO_1$ on $P$ is external to $\cO_2$. It follows that every external line to $\cO_1$ on $P$ is secant to $\cO_2$. 

Let $f$ be an o-polynomial over $\GF(q)$. A {\em generalized $f$-fan} in $\PG(2,q)$ is a set $\{\cO_s:s \in \GF(q)\}$ of ovals such that:

\begin{itemize}
\item[1.] every $\cO_s$ has nucleus $(0,1,0)$;
\item[2.] $\cO_s\cap \cO_t=\{(0,0,1)\}$, for all $s\neq t$;
\item[3.] $\cO_s$ and $\cO_t$ are compatible at $P_{st}=(0,1,(f(s)+f(t))/(s+t))$, for all $s\neq t$.
\end{itemize}

The connection between spreads of $T_2(\cD(f))$ and generalized $f$-fans in $\PG(2,q)$ is made explicit  in the following result.
\begin{theorem}\cite[Theorem 1.1]{boppr1}\label{th_1}
Let $f$ be an o-polynomial over $\GF(q)$ and $\pi_\infty$ the plane of $\PG(3,q)$ with equation $x_0=0$. Let $\Sigma$ be a spread of $T_2(\cD(f))$ containing $(0,0,0,1)\in\cD(f)\subset \pi_\infty$, and $\pi$ the plane of $\PG(3,q)$ with equation $x_1=0$. For $s\in\GF(q)$, let	$\cK_s$	be	the	set	of the $q$	lines	of	$\Sigma$	containing	$(0,1, s,f(s))$	and	let $\cO_s=\{\pi\cap \ell:\ell\in\cK_s\}\cup \{(0,0,0,1)\}$. Then, $\{\cO_s:s\in\GF(q)\}$ is a generalized $f$-fan of ovals in $\pi$ (with nucleus $(0,1,0)$). Conversely, if $\cF=\{\cO_s:s\in\GF(q)\}$ is a generalized $f$-fan of ovals in $\pi$ (with nucleus $(0,1,0)$), then 
\[
\Sigma(\cF)=\{(0,0,0,1)\}\cup\{\<X,(0,1,s,f(s))\>:s\in\GF(q), X\in \cO_s\setminus\{(0,0,0,1)\}\}
\]  
is a spread of $T_2(\cD(f))$  containing $(0,0,0,1)$.
\end{theorem}
 Two generalized fans $\cF$ and $\cG$ are said to be {\em projectively equivalent} if there exists a collineation of $\PG(2,q)$ mapping the set of the ovals of $\cF$ to the one of $\cG$. Note that there are spreads of $T_2(\cO)$ arising from  projectively equivalent generalized fans that are not equivalent, and there are inequivalent generalized fans that yield equivalent spreads \cite[Section 3.3]{boppr1}. From here the need to introduce the following definition derives: two generalized fans $\cF$ and $\cG$ are said to be {\em spread equivalent} if the associated spreads are equivalent.
%

%
A {\em posy of ovals} in $\PG(3,q)$ is a set of $q+1$ ovals $\cF=\{\cO_s:s\in \GF(q)\}\cup\{\cO\}$ such that:
\begin{itemize}
\item[1.] $\cO=\cD(f)$ is contained in $\pi_\infty:x_0=0$, for some o-polynomial $f$ over $\GF(q)$;
\item[2.] $\{\cO_s:s\in\GF(q)\}$ is a generalized $f$-fan of ovals in $\pi:x_1=0$.
\end{itemize}

We call the oval $\cO$ the {\em director oval} of the posy. 

Two posies $\cF$ and $\cF'$ are said to be {\em equivalent} if there exists a collineation of $\PG(3,q)$ mapping the director oval of $\cF$ to the one of $\cF'$ and the generalized fan of $\cF$ to the one of $\cF'$. 
Theorem \ref{th_1} and Corollary \ref{cor_2} allow us to extend naturally the correspondence between generalized fans in $\pi:x_1=0$ and spreads of Tits quadrangles to a correspondence between posies of $\PG(3,q)$ and pairs $(T_2(\cO), \Sigma)$, where $\Sigma$ is a spread of $T_2(\cO)$.

Two pairs $(T_2(\cO),\Sigma)$ and $(T_2(\cO'),\Sigma')$  are  {\em equivalent} if there exists a collineation of $\PG(3,q)$ mapping $\cO$ to $\cO'$ and $\Sigma$ to $\Sigma'$.

\begin{theorem}\label{th_5}
{Let  $\cO=\cD(f)$ and $\cO'=\cD(f')$ be two ovals in $\PG(2,q)$. Let $\Sigma$ (resp. $\Sigma'$) be a spread of  $T_2(\cO)$ (resp. $T_2(\cO')$) containing the point $(0,0,0,1)$,  and $\cF$ (resp.  $\cF'$) the posy arising from $(T_2(\cO), \Sigma)$ (resp. $(T_2(\cO'), \Sigma')$).}
 Then, 
\begin{itemize}
\item[(i)] if $\cO$ is not a conic, then $\cF$ and $\cF'$ are equivalent  if and only if  $(T_2(\cO),\Sigma)$ and $(T_2(\cO'),\Sigma')$ are equivalent;
\item[(ii)] if $\cO$ is a conic, then $\cF$ and $\cF'$ are equivalent   if and only if  $(T_2(\cO),\Sigma)$ and $(T_2(\cO'),\Sigma')$ are equivalent under an isomorphism from $T_2(\cO)$ to $T_2(\cO')$ mapping $(\infty)$ to $(\infty)'$.
\end{itemize}
\end{theorem}
\begin{proof}
(i) Assume that $\cF$ and $\cF'$ are equivalent. Then, there is a collineation $\phi$ of $\PG(3,q)$ taking $\cO$ to $\cO'$ and the generalized fan of $\cF$ to the one of $\cF'$.  Hence, by Corollary \ref{cor_2}\,(i), $T_2(\cO)$ is isomorphic to $T_2(\cO')$. By Theorem \ref{th_1}, the oval $\cO_s\in\cF$ is associated with the point $(0,1,s,f(s))\in\cO$, for all $s\in\GF(q)$. Let $\phi(\cO_s)=\cO'_{s'}\in\cF'$. We may assume  $\phi(\cO_0)=\cO'_0\in\cF'$ from Theorem 3.5 in \cite{boppr1}. 
 Note that the ovals $\cO_s$, $s\neq0$, and $\cO_0$ are compatible at $(0,0,1,f(s)/s)$. As the compatibility  is invariant under collineations of $\PG(3,q)$, $\cO'_{s'}$ and $\cO'_0$ are compatible at $\phi((0,0,1,f(s)/s))=(0,0,1,f'(s')/s')$, giving that $\cO'_{s'}$ is the oval associated with the point $(0,1,s',f'(s'))\in\cO'$. This implies that $\Sigma$ and $\Sigma'$ are equivalent under $\phi$.
Conversely, let $\Sigma$ be a spread of $T_2(\cO)$, $\cO$ not a conic, and $\Sigma'$ a spread of $T_2(\cO')$, such that there is an isomorphism $\phi$ from  $T_2(\cO)$ to  $T_2(\cO')$ taking $\Sigma$ to $\Sigma'$. By Corollary \ref{cor_2}\,(i), $\phi(\cO)=\cO'$, and the fan of $\cF$ is mapped via $\phi$ to the fan  of $\cF'$ since the compatibility is preserved by the collineations of $\PG(3,q)$.

(ii) Here, we proceed as in (i), provided that Corollary \ref{cor_2}\,(ii) is applied to both ``if'' and ``only if '' statements. 
\end{proof}



%
For any  generator $\alpha$ of  $\Aut(\GF(q))$, the {\em $\alpha$-cone}  $\Gamma_\alpha$ is defined to be the set of points $\{(x_0,x_1,x_2,x_3):x_1^\alpha=x_0x_2^{\alpha-1}\}$ of $\PG(3,q)$ (here, $0^{\alpha-1}= 0$ by convention). Therefore, $\Gamma_\alpha$ is a cone having as vertex $(0,0,0,1)$ and as base an oval equivalent to $
\cD(\alpha)$. An {\em $\alpha$-flock} of $\Gamma_\alpha$ is a set of $q$ planes partitioning $\Gamma_\alpha$ minus  the vertex into disjoint ovals, which turn out to be all equivalent to $\cD(\alpha)$. This definition was introduced by Cherowitzo \cite{che} as a generalization of flocks of the quadratic cone, which correspond to 2-flocks. 
\begin{proposition}\cite[Corollary 4.2]{boppr1}\label{prop_1}
The set of planes $\{a_t x_0+b_tx_1+c_t x_2+x_3=0:t\in\GF(q)\}$ of $\PG(3,q)$ is an $\alpha$-flock if and only if  $\{b_t x_0+a_tx_1+c_t x_2+x_3=0:t\in\GF(q)\}$ is a $1/\alpha$-flock.
\end {proposition}
Any $\alpha$-flock can be mapped to an $\alpha$-flock of the form $\{a_t x_0+t^{1/\alpha}x_1+c_t x_2+x_3=0:t\in\GF(q)\}$ by a collineation of $\PG(3,q)$ fixing $\Gamma_\alpha$. Such an $\alpha$-flock is said to be {\em normalized}.  There exists a link between normalized $\alpha$-flocks and certain generalized $f$-fans as described in the following theorem.
\begin{theorem}\cite[Theorem 1.2]{boppr1}\label{th_3}
Let  $\alpha$ be a generator of $\Aut(\GF(q))$. Let  $\cO_s=\{(1,t,t^{\alpha}+g(s)):t\in\GF(q)\}\cup \{(0,0,1)\}$. Then, $\{\cO_{s}:s \in \GF(q)\}$ is a generalized $f$-fan of ovals in $\PG(2,q)$ if and only if $\{f(t)x_0+t^{\alpha}x_1+g(t)x_2+x_3=0:t\in\GF(q)\}$ is a $1/\alpha$-flock.
\end{theorem}

A generalized fan of translation ovals is said to be {\em axial} if the common tangent $x_0=0$ is an axis for at least one of the ovals. Note that the generalized fan  $\{\cO_s:s\in\GF(q)\}$ in the previous theorem is an axial  oval set, where the common tangent $x_0 =0$ is an axis for all ovals in the fan. Another equivalence turns  out, the one between axial generalized fans of translation ovals and those fans arising from normalized $\alpha$-flocks.
\begin{theorem}\cite[Theorem 4.6]{boppr1}\label{th_2}
A generalized fan of translation ovals is axial if and only if it arises from a normalized  $\alpha$-flock,  for some  generator $\alpha$ of $\Aut(\GF(q))$.
\end{theorem}

Let $q$ be any prime power. A (normalized) $q$-{\em clan} is a set $\cC=\{A_t:t\in\GF(q)\}$ of $q$ $2\times 2$ matrices $A_t=\begin{pmatrix}a_t & t^{1/2} \\ 0 & b_t  \end{pmatrix}$ over $\GF(q)$, such that $A_0$ is the zero matrix and $A_s-A_t$ is anisotropic, for all $s\neq t$, that is $\u(A_s-A_t)\u^T=0$ if and only if $\u=0$. It is known that the set $\cC$ is a (normalized) $q$-clan if and only if $\cF(\cC)=\{a_tX_0+t^{1/2}X_1+b_tX_2+X_3=0:t\in\GF(q)\}$ is a flock of the quadratic cone \cite{thas1}.
\\Two (normalized) $q$-clans $\cC=\{A_t:t\in\GF(q)\}$ and $\cC'=\{A'_t:t\in\GF(q)\}$ are {\em equivalent} if there are $\lambda\in\GF(q)\setminus\{0\}$, $B\in\GL(2,q)$, $\sigma\in\Aut(\GF(q))$ and a permutation $\pi:t\mapsto  \bar t$ of $\GF(q)$ such that  
\[
A'_{\bar t}=\lambda BA_t^\sigma B^T+ A'_{\bar 0},
\]
for all $t\in\GF(q)$.  The definition of $q$-clan can be generalized to that of $\alpha$-clan, with $\alpha$ any automorphism of $\GF(q)$; this was done in \cite{che}.
\\Any given (normalized) $q$-clan $\cC$ gives rise to a generalized quadrangle $\mathrm{GQ}(\cC)$  of order $(q^2,q)$ which is an  elation GQ about the (elation) point $(\infty)$. For more details on $\mathrm{GQ}(\cC)$  we refer the reader to  \cite[A.1.1, A.1.2]{pt} and \cite{pay96}. Because of the connection between $q$-clans and flocks of the quadratic cone, a $\mathrm{GQ}(\cC)$ is called a {\em flock quadrangle}. 
\begin{theorem}\cite{pay96}\label{th_13} {\em(The Fundamental theorem of $q$-clan geometry)}
For any prime power $q$, let $\cC$ and
$\cC'$ be two (normalized) $q$-clans. Then, the following are equivalent: 
\begin{itemize}
 \item[(i)] $\cC$ and $\cC'$ are equivalent.
\item[(ii)] The flocks of the quadratic cone $\cF(\cC)$ and $\cF(\cC')$ are projectively equivalent.
\item[(iii)] $\mathrm{GQ}(\cC)$ and $\mathrm{GQ}(\cC')$ are isomorphic by an isomorphism mapping $(\infty)$ to $(\infty)$, $[A(\infty)]$ to $[A'(\infty)]$, and $(\underline 0, 0, \underline0)$ to $(\underline0, 0, \underline0)$\footnote{Here, we adopt the notation from \cite{pay96}: $[A(\infty)]$ is the set of lines of $\mathrm{GQ}(\cC)$ through $(\infty)$, and $\underline 0=(0,0)\in \GF(q)^2$}.
\end{itemize}
\end{theorem}
\begin{remark}\label{rem_3}
{\em 
We emphasize  that  a flock quadrangle gives rise to an equivalence class of flocks of the quadratic cone for each orbit on the lines through $(\infty)$ \cite{blp}.
}
\end{remark}
Another correspondence that plays an important role in this paper is that between flock quadrangles and herds of ovals in $\PG(2,q)$. 

A {\em herd} of ovals in $\PG(2,q)$, $q\ge 4$ even, is a family of $q+1$ ovals (not necessarily distinct) $\{\cO_s:s\in\GF(q)\cup\{\infty\}\}$, each of which  has nucleus $(0,0,1)$, contains the points $(1,0,0)$, $(0,1,0)$, $(1,1,1)$, and such that
\begin{align*}
 \cO_\infty  =  \{(1,t,f_\infty(t)):t\in\GF(q)\}\cup \{(0,1,0)\}\\[0.1in]
 \cO_s    =  \{(1,t,f_s(t)):t\in\GF(q)\}\cup \{(0,1,0)\}, & \,\,\,\, s\in\GF(q),
\end{align*}
where
\[
f_s(t)=\frac{f_0(t)+\kappa sf_\infty(t)+s^{1/2} t^{1/2}}{1+\kappa s+s^{1/2}},
\]
for some $\kappa\in\GF(q)$ with $\Tr(\kappa)=1$; here and below, $\Tr$ stays for the absolute trace function of $\GF(q)$. We will use $\cH(f_0,f_\infty)$ to denote the herd defined by $f_0$ and $f_\infty$.
\begin{remark}\label{rem_4}
\em{In \cite{okp2}, the notation $\cD(f)$, with $f$ an o-permutation, is used for the oval $\{(1,t,f(t)):t\in\GF(q)\}\cup\{(0,1,0)\}$, with nucleus $(0,0,1)$. Since in the present paper, $\cD(f)$ is used for the oval $\{(1,t,f(t)):t\in\GF(q)\}\cup\{(0,0,1)\}$, we will use  $\cO(f)$ for the oval $\{(1,t,f(t)):t\in\GF(q)\}\cup\{(0,1,0)\}$ to avoid confusion. Note that, under the collineation $\tau:(x_0,x_1,x_2)\mapsto(x_0,x_2,x_1)$, $\cO(f)$ is mapped, with the current notation,  to $\cD(f^{-1})$, with nucleus $(0,1,0)$. In the following, for the sake of simplicity, depending on the context,  we will identify the ovals $\cO(f)$ and $\cD(f^{-1})$.} 
\end{remark}

In  \cite[Lemma 1]{okp2} an action of $\PGammaL(2,q)$, $q$ even, on the vector space of the functions over $\GF(q)$ is defined: it is called the {\em magic action}.  Let $\mathfrak F$ be the vector space of all functions $f:\GF(q)\rightarrow \GF(q)$ such that $f(0)=0$. It is well known that each element of $\mathfrak F$ can be written as a polynomial in one variable of degree at most $q-1$. Furthermore, if $f(x)=\sum{a_ix^i}$ and $\gamma\in\Aut(\GF(q))$, then  $f^\gamma(x)=\sum{a_i^\gamma x^i}$.

Any element $\psi=(A,\gamma)\in\PGammaL(2,q)$, with $A=\begin{pmatrix}a& b \\ c & d\end{pmatrix}$ and $\gamma\in\Aut(\GF(q))$, acts on $\mathfrak F$ by mapping $f$ to $\psi f$, where 
\[
\psi f(x)=|A|^{-1/2}\left[(bx+d)f^\gamma\left(\frac{ax+c}{bx+d}\right)+b\,x\,f^\gamma\left(\frac{a}{b}\right)+d\,f^\gamma\left(\frac{c}{d}\right)\right];\]
here, if $b=0$ then $bf^\gamma(a/b) = 0$ by convention (and similarly for $df^\gamma(c/d)$).

This action of $\PGammaL(2,q)$ on $\mathfrak F$ is called  {\em magic action}.

\begin{theorem}\cite[Theorem 4,Theorem 6]{okp2}\label{th_14}
The magic action preserves the set of o-permutations, and for $g=\psi f$ we have that  $\cO(f)$ and $\cO(g)$ are projectively equivalent under a collineation $\overline\psi$ of $\PG(2,q)$ induced by $\psi$. Conversely, if $\cO(f)$ and $\cO(g)$, with $f,g$ o-permutations, are equivalent under $\PGammaL(3,q)$, then there is $\psi\in\PGammaL(2,q)$ such that $\psi f\in\<g\>$.
\end{theorem}

Since the magic action preserves the set of o-permutations, we say that two  o-permutations $f$ and $g$ are {\em equivalent} if they correspond under the magic action, i.e., there exists $\psi\in \PGammaL(2,q)$ such that $\psi f\in\<g\>=\{\lambda g:\lambda\in\GF(q)\}$.

Two herds $\cH(f_0,f_\infty)$ and $\cH(f'_0,f'_\infty)$ are {\em isomorphic} if there exists $\psi\in\PGammaL(2,q)$ such that for all $s\in\GF(q)\cup \{\infty\}$ we have $\psi f_s\in \<f_t'\>$, for some $t\in\GF(q)$, where the induced map $s \mapsto t$ is a permutation of $\GF(q) \cup \{\infty\}$.
%

%

In \cite{cppr} it is proved that a herd of ovals $\cH(f_0,f_\infty)$ gives rise to a (normalized) $q$-clan $\cC=\{A_t:t\in\GF(q)\}$, with
\[
A_t= \begin{pmatrix}
f_0(t) & t^{{1/2}}\\
0 & \kappa f_\infty(t)
\end{pmatrix}.
\] 
Conversely, such a (normalized) $q$-clan $\cC$ is shown to correspond to the  herd  of ovals $\cH(\cC)=\cH(f_0,f_\infty)$. 
\\In light of Theorem 16 in \cite{okp2}, for $q$ even, we can re-write the Fundamental Theorem enriched with the following equivalence:
\begin{quote} \emph{The herds $\cH(\cC)$ and $\cH(\cC')$ are isomorphic if and only if $\mathrm{GQ}(\cC)$ and $\mathrm{GQ}(\cC')$ are isomorphic by an isomorphism mapping $(\infty)$ to $(\infty)$ and $(\underline 0, 0, \underline0)$ to $(\underline0, 0, \underline0)$.}
\end{quote} 
\begin{remark}\label{rem_5}
{\em Let $\cH(\cC)$, $\cH(\cC')$ be two (not necessarily distinct) herds such that both the  automorphism groups of  $\mathrm{GQ}(\cC)$ and $\mathrm{GQ}(\cC')$ have just one  orbit on the lines through $(\infty)$. Then  $\cH(\cC)$ and $\cH(\cC')$ are isomorphic if and only if the flocks of the quadratic cone $\cF(\cC)$ and $\cF(\cC')$ are equivalent.
}
\end{remark}
All the above equivalences have been used to classify the spreads of $T_2(\cO)$, for $\cO$ an oval in $\PG(2,q)$, for $q=2,4,8,16,32$.  For $q=2,4$, the only ovals in $\PG(2,q)$ are conics \cite{segre57}, and so the classification of spreads of $T_2(\cO)$ is equivalent to the one of ovoids of $\PG(3,q)$: they are elliptic quadrics \cite{bose47,seid50}. The only ovals in $\PG(2,8)$ are  conics and  pointed conics \cite{segre57}. As in both cases  $T_2(\cO)$ is self-dual, to classify spreads of $T_2(\cO)$ is equivalent to classify ovoids in these GQs. This classification can be found in \cite[III.17.8]{pin}. The classification of spreads of $T_2(\cO)$ for $q=16,32$ is discussed in \cite{boppr}.    
\begin{theorem} 
All spreads of $T_2(\cO)$, for $\cO$ an oval in $\PG(2,q)$, $q$ even, are classified for $q\in\{2,4,8,16,32\}$.
\end{theorem}
In the next section, we provide the classification of spreads of $T_2(\cO)$ for $q=64$.

\section{The classification} \label{sec_3}
Our aim in this section is to classify, with the aid of a computer, all spreads of $T_2(\cO)$, for $\cO$ an oval in $\PG(2,64)$, by taking advantage of the equivalence between pairs $(T_2(\cO),\Sigma)$, $\Sigma$ a spread of $T_2(\cO)$, and posies of ovals in $\PG(3,64)$.

 Our programs, written in MAGMA \cite{magma}, were run on a macOS system with one quad-core Intel Core i7 2.2 GHz processor. All the computational operations take approximately 39 hours of CPU time. The code is available on email request to the third author.

In $\PG(2,64)$ there are, up to equivalence, exactly four hyperovals \cite{vander}: the regular hyperoval,  the Subiaco I hyperoval, the Subiaco II hyperoval, the Adelaide hyperoval, whose corresponding o-polynomials are those given in \cite[pp.10--11]{mps}. Precisely,  the regular hyperoval is defined by the o-polynomial  $f_L(x)=x^{1/2}$; for any $\delta\in\GF(q)$, with $\delta^2 + \delta + 1\neq 0$ and $\Tr(1/\delta) = 1$, the Subiaco I hyperoval is defined by the o-polynomial
\[
f_{S_I}(t)=\frac{\delta^2(t^4 +t)+\delta^2(1+\delta+\delta^2)(t^3 +t^2) }{(t^2 + \delta t + 1)^2}+t^{1/2},
\]
and the Subiaco II hyperoval is defined by 
\[
f_{S_{II}}(t)=\frac{f_{S_I}(t)+\kappa ag(t)+a^{1/2} t^{1/2}}{1+\kappa a+a^{1/2}},
\]
where $\kappa\in\GF(q)$ with absolute trace of $\kappa$ equal 1, $a\in\GF(q)$  with  $a^2+a+1=0$, and  
\[
g(t)=	\frac{\delta^4t^4 +\delta^3(1+\delta^2 +\delta^4)t^3 +\delta^3(1+\delta^2)t}{	(\delta^2 +\delta^5 +\delta^{1/2})(t^2 +\delta t +1)^2}	+ \frac{\delta^{1/2}}{(\delta^{2} +\delta^5 +\delta^{1/2})}t^{1/2};
\]
 the Adelaide hyperoval is defined by 
\[
f_A(t)=\frac{T(\beta^m)(t+1)}{T(\beta)}+\frac{T((\beta t+\beta^q)^m)}{T(\beta)(t+T(\beta)t^{1/2}+1)^{m-1}}+t^{1/2},
\]
where $m\equiv \pm\frac{q-1}{3}(\mod\, q+1)$, $\beta\in\GF(q^2)\setminus\{1\}$ with $\beta^{q+1}=1$ and $T(x)=x+x^q$, for all $x\in\GF(q^2)$.

Since every oval lies in a unique hyperoval,  all the ovals in $\PG(2,64)$, up to equivalence, can be obtained by removing one representative from each point-orbit of the stabilizer in $\PGammaL(3,64)$ of each hyperoval, considered  as a permutation group acting on it. This operation produces exactly 19 distinct ovals in $\PG(2,64)$ \cite{pen}. 

According to the definition of a generalized fan, and following the idea adopted in \cite{okpr}, for  $\cO$  any of the 19 ovals, and $t$ any tangent line to $\cO$,  we check if for each point $R$ on $t\setminus(\cO\cup\{N\})$ there exist an oval $\cO'$ in $\PG(2,64)$ that is compatible with $\cO$ at $R$. This implies that $\cO$ could be in a generalized $f$-fan such that $t$ is the common tangent to the ovals of the fan.

The compatibility condition can be relaxed by using the following definition. Consider two ovals $\cO_1$, $\cO_2$ and two points $P_1$, $P_2$ such that $P_i$ is not in $\cO_i$ and different from the nucleus of $\cO_i$, $i=1,2$. The pairs $(\cO_1,P_1)$ and $(\cO_2,P_2)$ {\em match} if there is a collineation $g$ for which $g(P_1)=P_2$ and $g(\cO_1)$ and $\cO_2$ are compatible at $P_2$.
 
The lines through any point $P$ in $\PG(2,q)$ can be parameterized with the elements of $\GF(q)\cup\{\infty\}$.  For each oval-point pair $(\cO,P)$, this parameterization is chosen so that $\infty$ is assigned to the tangent line to $\cO$ on $P$, and it is used in the following result to  characterize the matching condition. 
\begin{theorem}\cite[Theorem 3.3]{okpr}\label{th_18}
Let be $(\cO_1,P_1)$, $(\cO_2,P_2)$ be two oval-point pairs. Let $S_1$  be the set of parameters of the secant lines to $\cO_1$ through $P_1$, and $E_2$  the set of parameters of the  external lines to $\cO_2$ through $P_2$.   Then, the pairs $(\cO_1,P_1)$ and $(\cO_2,P_2)$  match if and only if there is $g\in\AGammaL(1,q)$ such that $gE_2=S_1$.
\end{theorem}
The set $S_{P_i}=S_i$ is known as the {\em local secant parameter set} associated with the pair $(\cO_i,P_i)$ \cite{okpr}; we refer to $E_{P_i}=E_i$  as the {\em local external parameter set} associated with the pair $(\cO_i,P_i)$. 

\begin{proposition}\label{prop_2}
The ovals in a generalized fan in $\PG(2,64)$  are all conics or all pointed conics.
\end{proposition}
\begin{proof}
By keeping in mind the above arguments and Theorem \ref{th_18}, we store just the pairs $(\cO,P)$ with $\cO$ being one  of the 19 ovals in $\PG(2,64)$ and $P$ varying among the representatives of the orbits of the stabilizer in $\PGammaL(3,64)$ of $\cO$ on the points not in $\cO$ and different from the nucleus. Such pairs are 45107 in number. For each such a pair $(\cO,P)$ we calculate the sets $S_P$ and $E_P$, and for both of them  we store the lexicographically least local parameter set in their respective orbit under $\AGammaL(1,64)$. Therefore,  the given pairs $(\cO_1,P_1)$ and $(\cO_2,P_2)$  match if and only if the lexicographically least local parameter set of $S_{1}$ is equal to the one of $E_2$.

Let $\cO$ be one of the 19 ovals in $\PG(2,64)$. To check if an oval, equivalent to $\cO$, could be in a generalized fan, we need to find a tangent of $\cO$ such that for each of its points $P$, not the nucleus and not in $\cO$, there is an oval-point pair matching with $(\cO,P)$.  For any representative $R$ from each point-orbit of the stabilizer in $\PGammaL(3,64)$ of $\cO$, considered  as a permutation group acting on it, let $t_R$ be the tangent to $\cO$ at $R$. For all pairs $(\cO,P)$, with $P\neq R$ on $t_R$ not the nucleus of $\cO$, we check if $(\cO,P)$ has a matching. Only the conic and the pointed conic survive this procedure. If $\cO$ is the conic, then $R$ can be chosen to be any point of $\cO$ since the stabilizer of $\cO$ is transitive on its points; so we may take $R=(1,0,0)$. If $\cO$ is the pointed conic, then the only point $R\in\cO$ that passes the above test is $(0,1,0)$. This implies that, a priori, a generalized fan consists just of conics and/or pointed conics; in any case, for such a fan the common tangent is the line $x_0=0$.
Furthermore, such a fan is anyway axial. In fact, it is known that each tangent to a conic is an axis, and the tangent $x_0=0$ is the unique axis of the pointed conic $\cD(1/2)$.  

Since both the conic and the pointed conic are translation ovals, from Theorem \ref{th_2}, every generalized fan in $\PG(2,64)$ arises from a normalized $\alpha$-flock with $\alpha\in\{2,\frac12\}$, $\alpha$  being a generator of $\Aut(\GF(64))$. This means that a generalized fan in $\PG(2,64)$ has the form of an oval set as in Theorem \ref{th_3}, with $\alpha\in\{2,\frac12\}$, that is it consists of either all conics or all pointed conics. 
\end{proof}
We first consider spreads of $T_2(\cD(f))$ containing the point $(0,0,0,1)$, with $\cD(f)$ an oval of $\PG(2,64)$. The following classification result will be often used in the sequel.
\begin{theorem}\cite{mps}\label{rem_6}
In $\PG(3,64)$ there are, up to equivalence,  exactly three  flocks of the quadratic cone:  the linear flock, the Subiaco flock, and the Adelaide flock.
\end{theorem}
 
Corresponding to these flocks there are three $q$-clans, $q=64$, that here we denote as in \cite{mps}: the {\em classical} $q$-clan $\cC_L$, the {\em Subiaco}  $q$-clan $\cC_S$ and the {\em Adelaide} $q$-clan $\cC_A$. 

\begin{remark}\label{rem_11}
\emph{As introduced in the previous section, we will denote by $\cH(\cC_\varepsilon)$, with $\varepsilon\in\{L,S,A\}$, the herd defined by the 64-clan $\cC_\varepsilon$.
We emphasize that $\cH(\cC_L)$ consists of $q+1$ copies of the conic $\cO(x^{1/2})$. The herd  $\cH(\cC_S)$ contains two subsets of projectively equivalent ovals: they correspond to the two ovals with  nucleus $(0,0,1)$ respectively contained in the Subiaco I and in the Subiaco II hyperovals. 
This implies that every oval in $\cH(\cC_S)$ is projectively equivalent to either  $\cO(f_{S_I})$ or to $\cO(f_{S_{II}})$. As a consequence, we have that the Subiaco flock is equivalently defined by o-polynomials representing non-equivalent ovals in $\PG(2,64)$. Finally, all ovals in $\cH(\cC_A)$ are projectively equivalent to the oval  $\cO(f_A)$, with  nucleus $(0,0,1)$,  contained in the Adelaide hyperoval.}
\end{remark}

In the following, for the sake of simplicity, $h(\cC)$ and $h^{-1}(\cC)$  will be used  to denote the set of o-polynomials  over $\GF(64)$ defining the ovals in the herd $\cH(\cC)$ and the set of their inverses, respectively. 
 Note that, for every o-polynomial $f$, $f^{-1}$ is also an o-polynomial: $\cD(f^{-1})$ is projectively equivalent to the oval $(\cD(f)\setminus\{(0,0,1)\})\cup\{(0,1,0)\}$.  

From now on, we assume $q=64$.
\begin{theorem}\label{th_6}
If $T_2(\cD(f))$ has a spread containing $(0,0,0,1)$ which  arises from a generalized $f$-fan of pointed conics, then $f\in h(\cC_L) \cup h(\cC_S)\cup h(\cC_A)$, up to equivalence. 
\end{theorem}
\begin{proof}
Assume $\Sigma$  is a spread of $T_2(\cD(f))$ containing $(0,0,0,1)$, arising from the  generalized $f$-fan of pointed conics $\cO_s=\{(1,0,t,t^{1/2}+g(s)):t\in\GF(q)\}\cup\{(0,0,0,1)\}$, for some permutation $g$ of $\GF(q)$. Then $\alpha=1/2$, so, by Theorem \ref{th_3}, the corresponding $1/\alpha$-flock is actually a flock $\cF(\cC)$ of the quadratic cone in $\PG(3,64)$, 
where $\cC$ is the $q$-clan defined by $f$ and $g$.

By taking into account Theorem \ref{rem_6}, the flock $\cF(\cC)$ is projectively  equivalent to one of the above three flocks. Since for $q=64$ two flocks $\cF(\cC)$ and $\cF(\cC')$ are projectively equivalent if and only if the herds $\cH(\cC)$ and $\cH(\cC')$ are isomorphic \cite[Theorems 2.1 and  2.6]{mps}, we see that, up to equivalence,  $f\in h(\cC_L) \cup h(\cC_S)\cup h(\cC_A)$. 
\end{proof}
\begin{remark}\label{rem_9}
{\em We stress that, for $f\in h(\cC)$  the oval in the corresponding herd is not $\cD(f)$ but it is $\cO(f)$, equivalent to $\cD(f^{-1})$ by   Remark \ref{rem_4}.
}
\end{remark}
\begin{theorem}\label{th_17}
The Tits quadrangle $T_2(\cD(f))$, with $f\in h(\cC_L)\cup h(\cC_S)\cup h(\cC_A)$, has a unique spread containing $(0,0,0,1)$,  arising from a generalized $f$-fan of pointed conics. 
\end{theorem}
\begin{proof}
Since $f$, together with a suitable o-permutation $g$, defines a flock of the quadratic cone in  $\PG(3,q)$, $T_2(\cD(f))$ has spreads containing $(0,0,0,1)$ arising from a generalized $f$-fan of pointed  conics $\cO_s=\{(1,0,t,t^{1/2}+g(s)):t\in\GF(q)\}\cup\{(0,0,0,1)\}$, according to Theorems \ref{th_3} and \ref{th_1}.
We note that, if $g'$ is another polynomial defining, together with $f$, an  equivalent flock, then the generalized $f$-fan of pointed  conics defined  by $g'$ is precisely $\{\cO_s:s \in \GF(q)\}$, because $g'$ is actually a permutation of $\GF(q)$, as is $g$. \end{proof}
%
%
%
\begin{remark}\label{rem_8}
{\em Note that the set $h(\cC_L)$ consists just of the o-polynomial $x^{1/2}$, which is fixed by the magic action. Furthermore, the o-polynomials in $h(\cC_S)$ are equivalent either to $f_{S_I}$ or to $f_{S_{II}}$, and the o-polynomials in $h(\cC_A)$ are all equivalent  to $f_{S_A}$.
}
\end{remark}
In view of the next results, for the sake of completeness, we list the number  of the point-orbits, together with their respective length, for the stabilizer $G_{f}$ in $\PGammaL(3,q)$ of each oval in the herds $\cH(\cC_S)$ and $\cH(\cC_A)$. By Remark \ref{rem_9}, such an oval is described by  $\cD(f)$, with $f\in\{f_{S_I}^{-1}, f_{S_{II}}^{-1}, f_{A}^{-1}\}$. In the case $f(x)=x^2$, i.e., $\cD(f)$ is a conic, it is well known that $G_{f}$ is transitive on the points of the conic. For the remaining stabilizers, we have:
\begin{itemize}\label{orbits}
\item[-] $G_{f_{S_{I}}^{-1}}$ has one orbit of length 5 and one of length 60, the latter containing $(0,0,1)$;
\item[-] $G_{f_{S_{II}}^{-1}}$ has one orbit of length 5  and four orbits of length 15, one of the latter containing $(0,0,1)$;
\item[-] $G_{f_{A}^{-1}}$ has one orbit of length 1, one orbit of length 4, which contains  $(0,0,1)$, and five orbits of length 12.
\end{itemize} 
Note that $G_f$, with $f\in\{f_{S_I}^{-1}, f_{S_{II}}^{-1}, f_{A}^{-1}\}$, is precisely  the stabilizer of the hyperoval $\cD(f)\cup\{(0,1,0)\}$, with the point $(0,1,0)$ being a singleton orbit of $G_f$.
Let $\{P_0,P_1,\ldots, P_r\}$, $P_0=(0,0,1)$, be a transversal of the orbits of $G_f$ on the points of $\cD(f)$.
The following computational result is the keystone of our classification: 
\begin{quote}
{\bf Result 1.\ }{\em For every $f\in  h^{-1}(\cC_\varepsilon)$, $\varepsilon\in\{S,A\}$ and every $P\in\{P_1,\ldots, P_r\}$ there exists $\psi\in\PGammaL(2,q)$ such that $\psi f\in \<g\>$, with $g\in h^{-1}(\cC_\varepsilon)$,  and $\overline\psi$ mapping $P$ to a point of $\cD(g)$ in the orbit of $(0,0,1)$ under $G_g$.}
\end{quote}

%
\begin{theorem}\label{th_7}
If $T_2(\cD(f))$ has a spread containing $(0,0,0,1)$ which  arises from a generalized $f$-fan of conics, then $f\in h^{-1}(\cC_L)\cup h^{-1}(\cC_S)\cup h^{-1}(\cC_A)$, up to equivalence. 
 \end{theorem}
 \begin{proof}
 Let $\Sigma$  be a spread of $T_2(\cD(f))$ containing $(0,0,0,1)$, arising from the  generalized $f$-fan of conics $\cO_s=\{(1,0,t,t^{2}+g(s)):t\in\GF(q)\}\cup\{(0,0,0,1)\}$, for some o-permutation $g$ of $\GF(q)$. Then $\alpha=2$, so, by Theorem \ref{th_3}, the planes of the corresponding $1/2$-flock are $f(t)x_0+t^{2}x_1+g(t)x_2+x_3=0$, for  all $t\in\GF(q)$.
 
 
From Proposition \ref{prop_1}, such a $1/2$-flock corresponds to the flock  $\{t^{2}x_0+f(t)x_1+  g(t)x_2+x_3=0:t \in\GF(q)\}$ of the quadratic cone in $\PG(3,q)$. 
By following \cite{cppr}, the $q$-clan associated with this flock can be normalized to the form
\[
\cC=\left\{\begin{pmatrix}
\widehat f(t) & t^{1/2}\\
0 & \widehat g(t)
\end{pmatrix}:t \in\GF(q)\right\},
\]
where $\widehat f=\rho^{-1}\circ f^{-1} \circ \rho$ and $\widehat g=g\circ f^{-1} \circ \rho$, with $\rho$  the automorphism $x \mapsto x^{1/2}$ of $\GF(q)$. Hence, 
\[
f=\rho\circ{\widehat f}^{-1} \circ \rho^{-1}, \hspace{.3in} g=\widehat g\circ{\widehat f}^{-1} \circ \rho^{-1}.
\]
By using the same arguments as in the proof of Theorem \ref{th_6}, $\widehat f\in h(\cC_L)\cup h(\cC_S)\cup h(\cC_A)$,  up to equivalence.

For every o-polynomial $h$, we have $(\rho\circ h\circ \rho^{-1})(x)=(h(x^2))^{1/2}=h^{1/2}(x)$. So the polynomial $h^{1/2}$ is the image of $h$ under the semilinear transformation $\psi=(I_2,\rho)\in\PGammaL(2,q)$ acting on the  set of all o-polynomials as defined in \cite{okp2}.
Thus, for $h=\widehat f^{-1}$, we obtain that $f$ and $\widehat f^{-1}$ correspond under the magic action via $\psi$, that is, $f=(\widehat f^{-1})^{1/2}$. We now prove that, if $\widehat f\in h(\cC_\varepsilon)$, up to equivalence, with $\varepsilon\in\{L,S,A\}$, then $\widehat f^{-1}$ is equivalent to a polynomial in  $h^{-1}(\cC_\varepsilon)$, giving $f\in h^{-1}(\cC_\varepsilon)$, up to equivalence.

If $\varepsilon=L$, then $\widehat f(x)=x^{1/2}$ because of the Remark \ref{rem_8}. Clearly, $\widehat f^{-1}\in h^{-1}(\cC_L)=\{x^2\}$.

Let $\varepsilon\in\{S,A\}$ and assume $\widehat f=\psi \tilde f$, for some $\tilde f\in h(\cC_\varepsilon)$.  Then, $\widehat f^{-1}=(\psi \tilde f)^{-1}$. Note that $(\psi \tilde f)^{-1}$ is an o-polynomial defining the oval $\cO((\psi\tilde f)^{-1})$, which is obtained from $\cO(\psi\tilde f)$ by swapping the nucleus $(0,0,1)$ with the point $(0,1,0)\in \cO(\psi\tilde  f)$. By Theorem \ref{th_14},  $\cO(\psi\tilde f)=\overline\psi\cO(\tilde f)$, for some collineation $\overline\psi$.
Let $O_P$ be the orbit of $P=\overline\psi^{-1}(0,1,0)\in\cO(\tilde f)$ under the stabilizer of $\cO(\tilde f)$ in $\PGammaL(3,q)$. By Result 1, there exists an o-polynomial $g\in h(\cC_\varepsilon)$ which is equivalent to $\varphi \tilde f$ via some element $\varphi\in\PGammaL(2,q)$, that is $\varphi \tilde f\in\<g\>$. 
Therefore, $\overline\varphi\cO(\tilde f)=\cO(g)$ and $\overline\varphi O_P$ is the orbit of $(0,1,0)$ under the  stabilizer of $\cO(g)$ in $\PGammaL(3,q)$. By swapping the nucleus $(0,0,1)$ of  $\cO(g)$ with $(0,1,0)$, we get the oval $\cO(g^{-1})$, with $g^{-1}\in h^{-1}(\cC_\varepsilon)$. By comparing the latter with the previous swapping, we can deduce that $\cO((\psi \tilde f)^{-1})$ and $\cO(g^{-1})$ are projectively equivalent, from which $(\psi \tilde f)^{-1}$ and $g^{-1}$ are equivalent by \cite[Theorem 6]{okp2}. This implies that $\widehat f^{-1}=(\psi \tilde f)^{-1}\in h^{-1}(\cC_\varepsilon)$, up to equivalence.
\end{proof}

\begin{theorem}\label{th_16}
The Tits quadrangle $T_2(\cD(f))$, with $f\in h^{-1}(\cC_L)\cup h^{-1}(\cC_S)\cup h^{-1}(\cC_A)$, has a spread containing $(0,0,0,1)$,  arising from a generalized $f$-fan of  conics. 
\end{theorem}
\begin{proof}
Let $f=\widehat f^{-1}$ with $\widehat f\in h(\cC_L)\cup h(\cC_S)\cup h(\cC_A)$. 
Then, $\{\widehat f(t)x_0+t^{1/2}x_1+\widehat g(t)x_2+x_3=0: t\in\GF(q)\}$ is a flock of the quadratic cone in $\PG(3,q)$, for some o-permutation $\widehat g$. From Proposition \ref{prop_1}, such a flock corresponds to the $1/2$-flock  $\{t^{1/2}x_0+\widehat f(t)x_1+\widehat  g(t)x_2+x_3=0:t \in\GF(q)\}$ of the cone $\Gamma_{1/2}$ in $\PG(3,q)$. By following \cite{che}, the $1/2$-clan associated with this $1/2$-flock can be normalized to the form
\[
\cC=\left\{\begin{pmatrix}
\tilde f(t) & t^2\\
0 & \tilde g(t)
\end{pmatrix}:t \in\GF(q)\right\},
\]
where $\tilde f=\widehat f^{-1/2}$ and $\tilde g=\widehat g\circ\widehat f^{-1} \circ \rho^{-1}$, with $\rho$  the automorphism $x \mapsto x^{1/2}$ of $\GF(q)$. The planes of the normalized $1/2$-flock defined by $\cC$ are $\tilde f(t)x_0+t^2x_1+\tilde  g(t)x_2+x_3=0$, for all $t \in\GF(q)$. By Theorems \ref{th_3} and \ref{th_1}, $T_2(\cD(\tilde f))$ has spreads containing $(0,0,0,1)$, arising from the generalized  $\tilde f$-fan of conics $\cO_s=\{(1,0,t,t^{2}+\tilde g(s)):t\in\GF(q)\}\cup\{(0,0,0,1)\}$.  By Theorem \ref{th_5}, we need to show that the posy $\cD(\tilde f)\cup \{\cO_s:s \in \GF(q)\}$ is equivalent to a posy consisting of the director oval $\cD(f)$ and a generalized $f$-fan of conics.

Let $\tau$ be the collineation of $\PG(2,q)$ defined in Remark \ref{rem_4}. Then,

\begin{align*}
\tau\cD(\tilde f)&=\{(1,t,\tilde f^{-1}(t)):t \in \GF(q)\}\cup\{(0,1,0)\}\\[.1in]
&=\{(1,t,\widehat f^{1/2}(t)):t \in \GF(q)\}\cup\{(0,1,0)\}\\[.1in]
&={\overline\psi}\cO(\widehat f)
=(\overline\psi\circ\tau)\cD(f),
\end{align*}
where  $\overline\psi=(I_3,\rho)\in \PGammaL(3,q)$. Hence, $\cD(f)$ is projectively equivalent to $\cD(\tilde f)$ via $\tau\circ\overline\psi\circ\tau=\overline\psi$. 

By hypothesis, the ovals $\cO_s$ and $\cO_t$ are compatible at $P_{st}=(0,1,(\tilde f(s)+\tilde f(t))/(s+t))$, for all $s\neq t$.
By applying the collineation $\overline\psi$ and taking into account that $\tilde f=f^{1/2}$ and $\tilde g=\widehat g\circ f\circ\rho^{-1}$, we obtain 
\[
\Omega_{s^2}=\overline\psi\cO_s=\{(1,0,x,x^2+(\rho^{-1}\circ\widehat g\circ f)(s^2)):x \in\GF(q)\}\cup\{(0,0,0,1)\},
\]
and the ovals $\Omega_{s^2}$ and $\Omega_{t^2}$ are compatible at $(0,1,(f(s^2)+f(t^2))/(s^2+t^2))$, for all $s\neq t$.

Thus, the ovals $\Omega_s$, $s\in\GF(q)$, together with $\cD(f)$, form the desired posy.
\end{proof}

\begin{remark}\label{rem_7}
{\em  The geometric interpretation of Proposition \ref{prop_1} is the following: given a normalized flock of the quadratic cone $\Gamma_2$, by replacing the generator $\<(0,0,0,1),(1,0,0,0)\>$ with the (nuclear) generator $\<(0,0,0,1),(0,1,0,0)\>$, we get the cone $\Gamma_{1/2}$, and any normalized flock of $\Gamma_2$  gives rise, up to normalization,  to a normalized flock of $\Gamma_{1/2}$, i.e., a normalized $1/2$-flock. On the other hand, by starting from a spread of $T_2(\cD(f))$ containing $(0,0,0,1)$, where $\cO(f^{-1})$ is an oval in the herd defining the flock, the above replacing operation produces a spread of $T_2(\cD(f^{-1}))$. By Theorems \ref{th_3} and \ref{th_1}, such a spread  contains the point $(0,0,0,1)$.  As already told, $\cD(f^{-1})$ is projectively equivalent to $(\cD(f)\setminus\{(0,0,1)\}) \cup \{(0,1,0)\}$, that is the oval obtained from $\cD(f)$ by applying nucleus swapping.}
\end{remark}
\begin{remark}\label{rem_10}
{\em  We recall that each of the 19 ovals in $\PG(2,64)$ is obtained  by removing one representative from a point-orbit of the stabilizer in $\PGammaL(3,64)$ of the hyperoval containing the oval itself. 
In particular, the oval $\cD(f)$, with $f\in\{f_L^{-1},f_{S_I}^{-1},f_{S_{II}}^{-1},f_A^{-1}\}$, arises from the removal of the point $(0,1,0)$, which turns out to be the nucleus of $\cD(f)$. 
In the light of  Result 1 and Remark \ref{rem_7}, all the remaining 15 ovals in $\PG(2,64)$ are described by $\cD(f)$, with $f$ in the set $h(\cC_L)\cup h(\cC_S)\cup h(\cC_A)$.
}
\end{remark}
The following theorem highlights that a flock of the quadratic cone can produce inequivalent ${1/2}$-flocks depending on the replaced generator. 
\begin{theorem}\cite[Theorem 2.1]{boppr}\label{th_10}
Each flock of the quadratic cone in $\PG(3,q)$ gives rise to as many inequivalent  ${1/2}$-flocks as the orbits of its stabilizer on generators of the quadratic cone. 
\end{theorem}
Let $\Phi$ be a flock of the quadratic cone $\Gamma_2$ in $\PG(3,q)$, defined by the o-permutations $f$ and $g$, with stabilizer $G$ in the collineation group of $\Gamma_2$. Let $O_L$ be the orbit of the generator $L=\<(0,0,0,1),(1,0,0,0)\>$ under the action of $G$ on the generators of $\Gamma_2$, and $\Sigma$ the spread of $T_2(\cD(f^{-1}))$ arising from the switching of $L$ with the nuclear generator. Let $L'$ be a generator not in $O_L$. 
As the collineation group of $\Gamma_2$ is transitive on the generators, there exists a collineation $\sigma$ mapping $L'$ to $\<(0,0,0,1),(1,0,0,0)\>$ and $\Phi$ to the equivalent flock $\Phi'=\sigma\Phi$. Let $\Phi'$ be defined by the o-permutations $f'$ and $g'$. 
If $\cD(f)$ and $\cD(f')$ are equivalent ovals, then the switching of $L'$ with the nuclear generator yields a spread $\Sigma'$ of $T_2(\cD(f^{-1}))$ which is inequivalent to $\Sigma$, by Theorem \ref{th_10}; if $\cD(f)$ and $\cD(f')$ are not equivalent, then the above switching yields a spread $\Sigma'$ of $T_2(\cD(f'^{-1}))$. On the other hand, by exchanging the roles of $\Phi$ and $\Phi'$ and considering  $\sigma^{-1}$ instead of $\sigma$, the same arguments apply. In conclusion, we have proved the following result.
%
\begin{corollary}\label{cor_3}
Let $\cD(f)$ be an oval in $\PG(2,64)$, with $f\in h^{-1}(\cC_L)\cup h^{-1}(\cC_S)\cup h^{-1}(\cC_A)$, and $\Phi$ be the flock of the quadratic cone defined by $f^{-1}$ together with a suitable o-permutation $g$. Then, $T_2(\cD(f))$ has as many inequivalent spreads containing $(0,0,0,1)$ arising from generalized fans of conics as the orbits of the stabilizer of $\Phi$ on generators of the quadratic cone.
\end{corollary}

%
%
It is known that the  stabilizer of the linear flock has one orbit on the generators of the quadratic cone. For the other flocks, we find the number of orbits on the generators  of the quadratic cone in $\PG(3,64)$ with the aid of a computer.
\begin{lemma} In $\PG(3,64)$ both the stabilizer of the Subiaco flock and the stabilizer of the Adelaide flock have seven orbits on the generators of the quadratic cone: one of length 1, one of length 4, and five  of length 12. 
\end{lemma}

By applying Corollary \ref{cor_3}, we get the following result.
\begin{theorem}\label{th_9}
Let $\cD(f)$ be an oval in $\PG(2,64)$, with  $f\in h^{-1}(\cC_L)\cup h^{-1}(\cC_S)\cup h^{-1}(\cC_A)$. Then, the number of inequivalent spreads of $T_2(\cD(f))$ containing $(0,0,0,1)$ which arise from generalized $f$-fans of conics is  one if $f=f_L^{-1}$ and seven if $f\in h^{-1}(\cC_S)\cup h^{-1}(\cC_A)$.
\end{theorem}

When $f\in h^{-1}(\cC_\varepsilon)$, $\varepsilon\in\{L,S,A\}$, then $T_2(\cD(f))$ is a proper subquadrangle of the flock quadrangle defined by the $q$-clan associated with the herd $\cH(\cC_\varepsilon)$ \cite{pay85}; we refer to \cite[Section 6.1]{boppr1} for an explicit description of $T_2(\cD(f))$ as a subGQ of the flock quadrangle. The following result describes the geometric nature of the subtended spreads of $T_2(\cD(f))$.
\begin{theorem}\cite[Theorem 6.2]{boppr1}\label{th_11}
If $T_2(\cD(f))$ is a subquadrangle of some flock GQ, then any generalized $f$-fan corresponding to a subtended spread of $T_2(\cD(f))$ consists entirely of conics. 
\end{theorem}
The converse of  the above theorem  also holds.
\begin{theorem}\label{th_12}
If $T_2(\cD(f))$  admits a spread $\Sigma$ containing $(0,0,0,1)$ and arising from a generalized $f$-fan of conics, then, up to isomorphim, $T_2(\cD(f))$ is a subGQ of a flock  GQ and $\Sigma$ is subtended.
\end{theorem}
\begin{proof}
Let $\Sigma$  be a spread of $T_2(\cD(f))$ containing $(0,0,0,1)$, arising from the  generalized $f$-fan of conics $\cO_s=\{(1,0,t,t^{2}+g(s)):t\in\GF(q)\}\cup\{(0,0,0,1)\}$, for some o-permutation $g$ of $\GF(q)$. Then, Theorem \ref{th_7} yields  $f\in h^{-1}(\cC_\varepsilon)$, with $\varepsilon\in\{L,S,A\}$. Whence, $\cD(f)$ is an oval in the herd $\cH(\cC_\varepsilon)$, and $T_2(\cD(f))$ is a proper subquadrangle of the flock quadrangle defined by $\cC_\varepsilon$.

By using the same arguments as in the proof of Theorem \ref{th_7}, we find that $f=\rho\circ\widehat f^{-1}\circ\rho^{-1}$ and $g=\widehat g\circ\widehat f^{-1}\circ\rho^{-1}$, for some $\widehat f$ and $\widehat g$ defining a (normalized) $q$-clan $\cC$. 
In addition, $\cD(f)$,  via the collineation $\tau$ defined in Remark \ref{rem_4}, can be identified with $\cO(f^{-1})=\cO(\widehat f^{1/2})$, which is  projectively equivalent to the oval $\cO(\widehat f)$ by \cite[Theorem 6]{okp2} under the collineation $\overline \psi=(I_3,\rho^{-1})\in \PGammaL(3,q)$. Hence, $\cD(f)$ is projectively equivalent to $\cO(\widehat f)$ under the collineation $\overline \psi \circ \tau$. 
In particular, by setting $s'=f(s)$, for $s\in \GF(q)$, the point $(1,s,f(s))$ of $\cD(f)$ is mapped to  $(1,s'^{2}, \widehat f(s'^{2}))$ of $\cO(\widehat f)$. The collineations $\tau$ and $\overline\psi$ of $\PG(2,q)$ naturally extend to $\PG(3,q)$, by setting $\tau:(x_0, x_1, x_2, x_3)\mapsto (x_0, x_1, x_3, x_2)$ and $\overline \psi=(I_4,\rho^{-1})$. Thus, we have 
\begin{align*}
\tau\cO_s&=\{(1,0,x^2+g(f^{-1}(s')),x):x \in \GF(q)\}\cup\{(0,0,1,0)\}\\[.1in]&=\{(1,0,x^2+\widehat g(s'^2),x):x \in \GF(q)\}\cup\{(0,0,1,0)\}.
\end{align*}

Now, put $k^2=x^2+\widehat g (s'^2)$, so that
\begin{equation}\label{eq_1}
\tau\cO_s=\{(1,0,k^2,k+ \widehat g^{1/2}(s')):k \in \GF(q)\}\cup\{(0,0,1,0)\}, 
\end{equation}
for all $s\in\GF(q)$. Finally,
\begin{equation}\label{eq_3}
(\overline \psi\circ\tau)\cO_s=\{(1,0,k^2,k+ \widehat g(s'^{2})):k \in \GF(q)\}\cup
\{(0,0,1,0)\}, 
\end{equation}
for all $s\in\GF(q)$.

By Corollary \ref{cor_2}, the collineation $\overline \psi \circ \tau$ provides an isomorphism from $T_2(\cD(f))$ to $T_2(\cO(\widehat f))$, which maps the spread $\Sigma$ to a spread $\widehat \Sigma$ of $T_2(\cO(\widehat f))$ containing $(0,0,1,0)$. The lines of $\widehat\Sigma$ are obtained by joining the point $(0,1,s'^{2}, \widehat f(s'^{2}))$ of $\cO(\widehat f)$ and the points of $(\overline \psi\circ\tau)\cO_s$, for all $s\in\GF(q)$.

By considering  the notation used in  \cite[Section 6]{boppr} and comparing the set of conics $\{(\overline \psi\circ\tau)\cO_s:s\in \GF(q)\}$ with the set of conics given at the end of the proof of Theorem 6.2 in \cite{boppr}, we see that $T_2(\cO(\widehat f))$ is the subquadrangle $\cS_{(1,0)}$ and the spread $\widehat \Sigma$ is subtended by the line $A(\infty)\cdot(\gamma,0,0)$, with $\gamma=(0,1)$; note that the function $g_s$ in \cite{boppr1} acts on $\GF(q)^2$ by mapping $\vv$ to $\vv A_s \vv^t$, where $A_s$ is a matrix in the $q$-clan defining the flock GQ.   This implies that $\widehat\Sigma$, and hence $\Sigma$, is subtended.                     
\end{proof}
It could be of some interest to point out that the second part of Theorem \ref{th_12} can be generalized  as follows:
\begin{theorem}
Let $q=2^h$ and  $T_2(\cD(f))$ be a subGQ of a flock  GQ. Then every spread $\Sigma$ of $T_2(\cD(f))$ arising from an $f$-fan of conics and containing $(0,0,0,1)$ is subtended.
\end{theorem}
\begin{proof}
The claim follows from the proof of the second part of Theorem \ref{th_12}, which does not make use of $q=64$.
\end{proof}
Now we will deal with spreads of $T_2(\cD(f))$, $\cD(f)$ an oval of $\PG(2,64)$, containing a point $P$ not $(0,0,0,1)$. 
\begin{theorem}\label{th_15}
If $T_2(\cD(f))$ has a spread containing $P$ different from $(0,0,0,1)$, then $f\in h(\cC_L)\cup h(\cC_S)\cup h(\cC_A)\cup h^{-1}(\cC_L)\cup h^{-1}(\cC_S)\cup h^{-1}(\cC_A)$, up to equivalence. 
\end{theorem}
\begin{proof}
Let $P\in\cD(f)\setminus\{(0,0,1)\}$ and $\Sigma$ a spread of $T_2(\cD(f))$ containing it. By Corollary \ref{cor_2}, any  collineation of $\PG(2,q)$ mapping $P$ to $(0,0,1)$ and $\cD(f)$ to $\cD(f')$ can be extended to a collineation of $\PG(3,q)$ mapping $T_2(\cD(f))$ to $T_2(\cD(f'))$ and $\Sigma$ to a spread $\Sigma'$ of $T_2(\cD(f'))$ containing $(0,0,0,1)$.

By Theorems \ref{th_6} and \ref{th_7}, the o-polynomial $f'$ belongs to $h(\cC_L)\cup h(\cC_S)\cup h(\cC_A)\cup h^{-1}(\cC_L)\cup h^{-1}(\cC_S)\cup h^{-1}(\cC_A)$, up to equivalence. By Remark \ref{rem_4} and Theorem \ref{th_14}, there exists $\psi\in\PGammaL(2,q)$ such that $\psi f\in\<f'\>$, giving $f\in h(\cC_L)\cup h(\cC_S)\cup h(\cC_A)\cup h^{-1}(\cC_L)\cup h^{-1}(\cC_S)\cup h^{-1}(\cC_A)$, up to equivalence. 
\end{proof}
By taking into account Remark \ref{rem_10}, an oval $\cO$ in $\PG(2,64)$ is projectively equivalent to $\cD(f)$ with $f\in  h(\cC_L)\cup h(\cC_S)\cup h(\cC_A)\cup h^{-1}(\cC_L)\cup h^{-1}(\cC_S)\cup h^{-1}(\cC_A)$. Thus, the following result gives  the complete classification of spreads of $T_2(\cD(f))$ with $\cD(f)$ an oval in $\PG(2,64)$:

\begin{theorem}\label{th_8}
\begin{itemize}
\item[(a)] If $f=f_L^{-1}$, i.e.,  $\cD(f)$ is a conic, then $T_2(\cD(f))$ admits  a unique spread on each point of $\cD(f)$. They are all equivalent and each one arises from a generalized fan of conics.
\item[(b)] If $f\in h^{-1}(\cC_S)\cup h^{-1}(\cC_A)$, then $T_2(\cD(f))$ admits seven inequivalent spreads on each point of $\cD(f)$, and each one arises from a generalized fan of conics.
\item[(c)]  If $f\in h(\cC_L)\cup h(\cC_S)\cup h(\cC_A)$, then $T_2(\cD(f))$ admits 
a unique spread on $(0,0,0,1)$, arising from a generalized fan of pointed conics. There is no spread on the other points of $\cD(f)$.
\end{itemize}
\end{theorem}
\begin{proof}

Theorems \ref{th_17} and \ref{th_9} describe completely the spreads of $T_2(\cD(f))$ containing $(0,0,0,1)$.

(a) When $f(x)=f_L^{-1}(x)=x^2$, $\cD(f)$ is a conic. Since the stabilizer in $\PGammaL(3,q)$ of $\cD(f)$ is transitive on its points, $T_2(\cD(f))$ admits  a unique spread on each point of $\cD(f)$. They are all equivalent and each one arises from a fan of conics. 

(b) Let $f\in h^{-1}(\cC_\varepsilon)$, $\varepsilon\in\{S,A\}$, and $P$ be a point of $\cD(f)$. If $P$ lies in the  orbit of $(0,0,1)$  under the action of the stabilizer in $\PGammaL(3,q)$ of $\cD(f)$, then $T_2(\cD(f))$ admits seven inequivalent spreads containing $P$ which arise from generalized fans of conics. 

Let now $P$ be  not in the orbit of $(0,0,1)$  under the action of the stabilizer in $\PGammaL(3,q)$ of $\cD(f)$. 
Our attempt is to reduce the study of spreads of $T_2(\cD(f))$ containing $P$ to the case of spreads of an isomorphic Tits quadrangle containing $(0,0,0,1)$, by using arguments from \cite{okp2}. With this in mind, for the sake of simplicity, we identify the ovals $\cD(f)$ and $\cO(f^{-1})$, as announced in Remark \ref{rem_4}.
 
Let $\cH(\cC_\varepsilon)$, $\varepsilon\in\{S,A\}$, be the herd containing $\cO(f^{-1})$ and $h(\cC_\varepsilon)$ be the corresponding set of o-polynomials, which contains $f^{-1}$. By Result 1, $f^{-1}$ is equivalent to some o-polynomial $g\in h(\cC_\varepsilon)$ in such a way there is a collineation in $\PGammaL(3,q)$  mapping $P$ to a point in the orbit of $(0,1,0)\in\cO(g)$ under the stabilizer of $\cO(g)$. 
%
From Theorem \ref{th_9} we may conclude that, for any point $P\in\cD(f)$, $f\in h^{-1}(\cC_S)\cup h^{-1}(\cC_A)$, the Tits quadrangle $T_2(\cD(f))$ has seven inequivalent spreads containing $P$, each one arising from a generalized  fan of conics.
\comment{
(c) If $f=f_L=x^{1/2}$, then $\cD(f)$ is a pointed conic, whose nucleus, say $R$, is a point of the conic $\cD(2)$.  Let $P\in \cD(f)\setminus\{(0,0,1)\}$, and assume that $T_2(\cD(f))$ has a spread on $P$. If we  consider a collineation of $\PG(2,q)$ that maps $P$ to $(0,0,1)$ and $\cD(1/2)$ to $\cD(f')$ fixing the nucleus $R$, it  follows that $T_2(\cD(f'))$ should have a spread, say $\Sigma$ on $(0,0,0,1)$. Since $\cD(1/2)$ and $\cD(f')$ are projectively equivalent, $\Sigma$ is defined by a generalized  $f'$-fan of pointed conics by Theorems \ref{th_6} and \ref{th_7}. Therefore, $\{t^{1/2}x_0+f'(t)x_1+g(t)x_2+x_3=0: t\in\GF(q)\}$, is a 1/2-flock, whose associated 1/2-clan can be normalized to
\[
\cC=\left\{\begin{pmatrix}
 (f'^{-1})^{1/2}(t) & t^2\\
0 & \widehat g(t)
\end{pmatrix}:t \in\GF(q)\right\},
\]
where $\widehat g$ is a suitable o-polynomial over $\GF(q)$. Theorem \ref{th_3} implies that 
 $\cO_s=\{(1,t,t^2+\widehat g(s)):s\in\GF(q)\}\cup\{(0,0,1)\}$ is a generalized $(f'^{-1})^{1/2}$-fan of conics. This yields that $T_2(\cD((f'^{-1})^{1/2}))$, and hence $T_2(\cD(f'^{-1}))$, has a spread at $(0,0,0,1)$ defined by a generalized $f'^{-1}$-fan of conics. From Theorem \ref{th_7}, $\cD(f'^{-1})$ is a conic. On the other hand, $\cD(f'^{-1})$ is obtained from $\cD(f')$ by swapping  the nucleus $R$ with the  point $P=(0,0,1)$, giving that $\cD(f'^{-1})$ is still a pointed conic. So we get a contradiction, and  what we assumed  is not possible. This means that $T_2(\cD(1/2))$ admits a spread only on $(0,0,0,1)$.
}

(c) Let $f\in h(\cC_\varepsilon)$, $\varepsilon\in\{L,S,A\}$.  Let $P$ be a point of $\cD(f)\setminus\{(0,0,1)\}$ and assume that $T_2(\cD(f))$ has a spread on $P$.  

Note that, via the nucleus swapping, $\cD(f)$ is mapped to $\cO(f)$, in such a way that $\cD(f)\setminus\{(0,0,1)\}=\cO(f)\setminus\{(0,1,0)\}$. Therefore, $P$ is a point of $\cO(f)$ different from $(0,1,0)$, and it clearly lies in one of the point-orbits under the stabilizer of $\cO(f)$ in $\PGammaL(3,q)$ (see page \pageref{orbits}). Recall that  $\cO(f)$ is an oval of $\cH(\cC_\varepsilon)$ with nucleus $(0,0,1)$.

Via the collineation $\tau$ of $\PG(2,q)$ as defined in Remark \ref{rem_4}, $\cD(f)^\tau=\cO(f^{-1})$, whose nucleus is $(0,0,1)$.  It turns out that $\cO(f)\setminus\{(0,1,0)\}=\cD(f)\setminus\{(0,0,1)\}$ is projectively equivalent to $\cO(f^{-1})\setminus\{(0,1,0)\}$ via the collineation $\tau$. Therefore, the point $P^\tau$ of $\cO(f^{-1})\setminus\{(0,1,0)\}$  corresponds to the point $P$ of $\cO(f)\setminus\{(0,1,0)\}$.

 We now distinguish two cases: $P$ lying in the same orbit as $(0,1,0)$ under the stabilizer of $\cO(f)$ in $\PGammaL(3,q)$, or not.

Let $P$ be a point of $\cO(f)\setminus\{(0,1,0)\}$ in the same orbit as $(0,1,0)$ under the stabilizer of $\cO(f)$ in $\PGammaL(3,q)$.  Let $\psi$ be a collineation of $\PG(2,q)$ that fixes the nucleus $(0,0,1)$ of $\cO(f^{-1})$, maps $P^\tau$ to $(0,1,0)$ and  $\cO(f^{-1})$ to $\cO(f')$, for some o-polynomial $f'$, which is  equivalent to $f^{-1}\in h^{-1}(\cC_\varepsilon)$. 
Under the collineation $\tau$ that maps  $\cO(f')$ to $\cD(f'^{-1})$, we may say that $T_2(\cD(f'^{-1}))$ has a spread containing $P^{\tau\psi\tau}=(0,0,0,1)$. This means that $T_2(\cD(f'))$, which is equivalent to $T_2(\cO(f'^{-1}))=T_2((\cD(f'^{-1})\setminus\{P^{\tau\psi\tau}\})\cup\{(0,0,1,0)\})$ by Corollary \ref{cor_2}, has a spread $\Sigma$ on $(0,0,0,1)$, with $f'\in h^{-1}(\cC_\varepsilon)$, up to equivalence.

 By Theorem \ref{th_6}, $\Sigma$ has to arise from a generalized $f'$-fan of conics. By Theorem \ref{th_3}, this is equivalent to saying that  $\{f'(t)x_0+t^2x_1+g(t)x_2+x_3=0:t\in\GF(q)\}$, for some o-permutation $g$, is a 1/2-flock. It corresponds to the flock $\{t^2x_0+f'(t)x_1+g(t)x_2+x_3=0:t\in\GF(q)\}$  of the quadratic cone by Proposition \ref{prop_1}, which can be normalized to $\{(f'^{-1})^2(t)x_0+t^{1/2}x_1+g'(t)x_2+x_3=0:t\in\GF(q)\}$. By virtue of the equivalence between flocks of the quadratic cone in $\PG(3,64)$ and herds in $\PG(2,64)$ \cite{mps},  the o-polynomial $(f'^{-1})^2$, and hence $f'^{-1}$, has to be equivalent to some o-polynomial in the set $h(\cC_\varepsilon)$ containing $f$, that is  $f'^{-1}\in h(\cC_\varepsilon)$, up to equivalence. This implies that $\cO(f'^{-1})$ is projectively equivalent to an oval in $\cH(\cC_\varepsilon)$.

 On the other hand, $\cO(f'^{-1})$ is obtained from $\cO(f')$ by first swapping the nucleus $(0,0,1)$ of $\cO(f')$ with $P^{\tau\psi}\in\cO(f')$ and then applying the collineation $\tau$. Observe that, taking into account the previous arguments, the hyperovals $\cO(f')\cup\{(0,1,0)^{\tau\psi}\}$ and $\cO(f'^{-1})\cup \{P^{\tau\psi\tau}\}$ are both projectively equivalent to the hyperoval $\cO(f)\cup\{(0,0,1)\}$. This implies that $\cO(f')$ is projectively equivalent to the oval $(\cO(f)\setminus\{(0,1,0)\})\cup\{(0,0,1)\}$, while $\cO(f'^{-1})$  to the oval $(\cO(f)\setminus\{P\})\cup\{(0,0,1)\}$.
  As $P$ and $(0,1,0)$ lie in the same orbit under the stabilizer of $\cO(f)$ in $\PGammaL(3,q)$,  the ovals $\cO(f'^{-1})$ and $\cO(f')$  are projectively equivalent. By Remark \ref{rem_11}, the oval $(\cO(f)\setminus\{(0,1,0)\})\cup\{(0,0,1)\}$ is not in any herd, giving that $f'$, and hence $f'^{-1}$, cannot be equivalent to some o-polynomial in any set $h(\cC_\varepsilon)$, and this yields a contradiction. Therefore, $T_2(\cD(f))$ has no spread in any point $P\in \cD(f)\setminus\{(0,0,1)\}$ lying in the same orbit as $(0,1,0)$ under the stabilizer of $\cO(f)$ in $\PGammaL(3,q)$.
 
If $f(x)=x^{1/2}$, then $\cO(f)$ is a conic, and it is well known that its stabilizer  in $\PGammaL(3,q)$ has a unique orbit on the points.  From this and the previous arguments, we may conclude that the Tits quadrangle $T_2(\cD(1/2))$ has no spread on points of $\cD(1/2)$ different from $(0,0,0,1)$.
 
Now, let $P$ be a point of $\cO(f)\setminus\{(0,1,0)\}$ not in the orbit of $(0,1,0)$ under the stabilizer of $\cO(f)$ in $\PGammaL(3,q)$. It turns out that $P^\tau$ is a point of  $\cO(f^{-1})$ different from $(0,1,0)$.  
Let  $\overline\psi$ be  a collineation that fixes the nucleus $(0,0,1)$ of $\cO(f^{-1})$, maps $P^\tau$ to $(0,1,0)$ and  $\cO(f^{-1})$ to $\cO(f')$, for some o-polynomial $f'$, which is equivalent to $f^{-1}\in h^{-1}(\cC_\varepsilon)$.  By arguing as above, we find  $f'^{-1}\in h(\cC_\varepsilon)$, up to equivalence. In particular, $f'^{-1}$ is equivalent to $f$.
  
 On the other hand,  $\cO(f'^{-1})$ is obtained from $\cO(f')$ by first swapping the nucleus $(0,0,1)$ of  $\cO(f')$ with $P^{\tau\bar\psi}=(0,1,0)\in\cO(f')$ and then applying $\tau$. This implies that  $\cO(f'^{-1})$ is equivalent to $(\cO(f)\setminus\{P\})\cup\{(0,0,1)\}$. But, as $P$ and $(0,1,0)$ lie in different orbits on $\cO(f)$, then $(\cO(f)\setminus\{P\})\cup\{(0,0,1)\}$, and  hence $\cO(f'^{-1})$, cannot be equivalent to $\cO(f)$, giving that $f'^{-1}$ is not equivalent to $f$; a contradiction. 

In conclusion, $T_2(\cD(f))$, with $f\in h(\cC_\varepsilon)$, with $\varepsilon\in\{L,S,A\}$, does not admit any spread on $P\in\cD(f)$ different from $(0,0,0,1)$.
\end{proof}

\begin{corollary}
Each spread of  $T_2(\cD(f))$, with $f\in h^{-1}(\cC_L)\cup h^{-1}(\cC_S)\cup h^{-1}(\cC_A)$, is subtended.
\end{corollary}

\begin{proof}
The result follows by extending Theorem \ref{th_12} to every point of $\cD(f)$ via Result 1.
\end{proof}

%
\begin{corollary}
The Tits quadrangle $T_2(\cD(f))$, with $f\in h(\cC_\varepsilon)$, $\varepsilon\in\{L,S,A\}$, is not a proper subquadrangle of a {\em GQ} of order $(2^{12},2^6)$.
\end{corollary}
\begin{proof}
By way of contradiction, suppose that $\cS=T_2(\cD(f))$ is a proper subquadrangle of a GQ $\cS'$ of order $(2^{12},2^{6})$. Let $P$ be any point of $\cD(f)$ and $\ell$ a line of $\cS$ skew with $P$. Then, $\ell$ and $P$ extend to lines, say $\ell'$ and $P'$, of  $\cS'$. Let $x$ be  any point of $\ell'\setminus \ell$  and $m'$ the unique line   such that $x\, {\rm I}\, m'\, {\rm I}\, y\, {\rm I}\,  P'$, for some point $y$ of $P'$. Then, $m'$ contains no points  of $\cS$ and each point incident with $m'$ is incident with a unique (extended) line in $\cS$. By \cite[2.2.1]{pt} this set of lines is a spread of $\cS$, which contains the line $P$ of $\cS$. This implies that every point $P$ of $\cD(f)$ is contained in a spread. But this contradicts Theorem \ref{th_8} (c). 
\end{proof}

In the light of the evidence for orders 16, 32, and 64, we venture the following strengthening of the
main theorem of \cite{brown}:

{\bf Conjecture.} If $\cO$ is  a pointed conic in  $\PG(2,q)$, $q\ge 16$ even, then   $T_2(\cO)$ is not a  proper subGQ of a GQ of order $(s,q)$.


\begin{thebibliography}{999}
%
\bibitem{blp} L.~Bader, G.~Lunardon, S.E.~Payne,  On $q$-clan geometry, $q=2^e$, {\em Bull. Belg. Math. Soc. Simon Stevin} {\bf 1} (1994),  301--328. 
%
\bibitem{bose47}  R.C.~Bose, Mathematical theory of the symmetrical factorial design, {\em Sankhy$\bar{a}$} {\bf 8} (1947), 107--166. 
%
\bibitem{brown} M. R.~ Brown, The determination of ovoids of $\PG(3,q)$ containing a pointed conic. Second Pythagorean Conference (Pythagoreion, 1999), J. Geom. {\bf 67} (2000), 61--72.
%
\bibitem{boppr1} M.R.~Brown, C.M.~O'Keefe, S.E.~Payne, T.~Penttila, G.F.~Royle, Spreads of $T_2(\cO)$, $\alpha$-flocks and ovals, {\em Des. Codes Cryptogr.} {\bf 31} (2004),  251--282. 
%
\bibitem{boppr} M.R.~Brown, C.M.~O'Keefe, S.E.~Payne, T.~Penttila, G.F.~Royle,  The classification of spreads of $T_2(\cO)$ and $\alpha$-flocks over small fields, {\em  Innov. Incidence Geom.} {\bf 6/7} (2007/08), 111--126.
%
\bibitem{che}  W.~Cherowitzo,  $\alpha$-flocks and hyperovals, {\em  Geom. Dedicata} {\bf 72} (1998), 221--246.
%
\bibitem{cppr} W.E.~Cherowitzo, T.~Penttila, I.~Pinneri, G.F.~Royle,  Flocks and ovals, {\em Geom. Dedicata} {\bf 60} (1996),  17--37.
%
\bibitem{dem} P.~Dembowski, {\em Finite geometries},  Springer-Verlag, Berlin-New York, 1968.
%
\bibitem{hir} J.W.~Hirschfeld, Projective Geometries over Finite Fields. Second Edition, Oxford University Press, Oxford, 1998.
%
%
\bibitem{magma} W.~Bosma, J.~Cannon, C.~Playoust, The Magma algebra system. I. The user language, {\em J. Symbolic Comput.} {\bf 24} (1997), 235--265. 
%
\bibitem{mps}  G.~Monzillo, T.~Penttila, A.~Siciliano,  Classification of flocks of the quadratic cone in $\PG(3,64)$, {\em Finite Fields and their Applications}, {\bf 81} (2022), 102035, 13 pp.
%
\bibitem{okp} C.M.~O'Keefe, T.~Penttila, Subquadrangles of generalized quadrangles of order $(q^2,q)$, $q$ even, {\em  J. Combin. Theory Ser. A} {\bf 94} (2001),  218--229.
%
\bibitem{okp2} C.M.~O'Keefe, T.~Penttila, Automorphism groups of generalized quadrangles via an unusual action of $\PGammaL(2,2^h)$, {\em European J. Combin.} {\bf 23} (2002), 213--232. 
%
\bibitem{okpr} C.M.~O'Keefe, T.~Penttila,  G.F.~Royle, Classification of ovoids in $\PG(3,32)$, {\em J. Geom.} {\bf 50} (1994), 143--150.
%
\bibitem{payne} S.E.~Payne, A complete determination of translation ovoids in finite Desarguian planes, Atti Accad. Naz. Lincei Rend. Cl. Sci. Fis. Mat. Nat. (8) {\bf 51} (1971), 328--331 (1972).
%
\bibitem{pay85}  S.E.~Payne, A new infinite family of generalized quadrangles, {\em Congr. Numer.} {\bf 49} (1985), 115--128.
%
\bibitem{pay96} S.E.~Payne, The fundamental theorem of $q$-clan geometry, {\em Des. Codes Cryptogr.} {\bf 8} (1996),  181--202. 
%
\bibitem{pm} S.E.~Payne, C.C.~Maneri, A family of skew-translation generalized quadrangles of even order. Proceedings of the thirteenth Southeastern conference on combinatorics, graph theory and computing (Boca Raton, Fla., 1982).
{\em Congr. Numer.} {\bf 36} (1982), 127--135. 
%
\bibitem{pt} S.E.~Payne, J.A.~Thas, Finite Generalized Quadrangles, second ed., EMS Ser. Lect. Math., European Mathematical Society, Z\"urich, 2009.
%
%
\bibitem{pen} T.~Penttila, Uniqueness of the inversive plane of order sixty-four, {\em  Des. Codes Cryptogr.} {\bf 90} (2022),  827--834.
%
\bibitem{pin} I.~Pinneri, {\em Flocks, generalized Quadrangles and Hyperovals}, Ph.D. Thesis, University of Western Australia, 1996.
%
\bibitem{segre57}  B.~Segre, Sui $k$-archi nei piani finiti di caratteristica due,  {\em Rev. Math. Pures Appl.} {\bf 2} (1957), 289--300.
%
%
\bibitem{seid50}  E.~Seiden, A theorem in finite projective geometry and an application to statistics, {\em Proc. Amer. Math. Soc.} {\bf 1} (1950), 282--286.
%
\bibitem{thas72} J.A.~Thas, Ovoidal translation planes, {\em Arch. Math.} {\em 23} (1972), 110--112. 
%
\bibitem{thas1} J.A.~Thas, Generalized quadrangles and flocks of cones, {\em European J. Combin.} {\bf 8} (1987),  441--452.
%
\bibitem{tp} J.A.~Thas,  S.E.~Payne,  Spreads and ovoids in finite generalized quadrangles. Geometriae Dedicata, {\bf 52} (1994), 227--253.
%
\bibitem{tits} J.~Tits, Ovo\"ides et groupes de Suzuki, {\em Arch. Math.} {\bf 13} (1962), 187--198.
%
\bibitem{vander} P.~Vandendriessche, Classification of the hyperovals in $\PG(2,64)$, {\em  Electron. J. Combin.} {\bf 26} (2019), Paper No. 2.35, 12 pp.
%
%
\end{thebibliography}
\end{document}